\documentclass{amsart}
\usepackage{amsfonts,amssymb,amscd,amsmath,enumerate,verbatim,calc}
\usepackage[all]{xy}
\usepackage{hyperref}

\newcommand{\CM}{Cohen-Macaulay }
\newcommand{\LC}{local cohomology module }

\newcommand{\wrt}{with respect to}

\newcommand{\n}{\mathfrak{n} }
\newcommand{\m}{\mathfrak{m} }

\newcommand{\Z}{\mathbb{Z} }
\newcommand{\CC}{\mathbb{C} }
\newcommand{\N}{\mathbb{N} }

\newcommand{\GG}{\mathbb{G}}
\newcommand{\EE}{\mathbb{E}}

\providecommand{\Q}{{\mathbb Q}}

\newcommand{\rt}{\rightarrow}

\newcommand{\im}{\operatorname{image}}

\newcommand{\Ass}{\operatorname{Ass}}

\newcommand{\depth}{\operatorname{depth}}

\newcommand{\ann}{\operatorname{ann}}
\newcommand{\hgt}{\operatorname{height}}

\providecommand\soc{\text{\rm soc}}
\providecommand\Mod{\text{\rm Mod}}

\newcommand{\Supp}{\operatorname{Supp}}
\providecommand\Spec{\text{\rm Spec}}
\newcommand{\projdim}{\operatorname{projdim}}
\newcommand{\injdim}{\operatorname{injdim}}

\newcommand{\Hom}{\operatorname{Hom}}
\newcommand{\Ext}{\operatorname{Ext}}

\theoremstyle{plain}

\newtheorem{theorem}{Theorem}[section]

\newtheorem{lemma}[theorem]{Lemma}
\newtheorem{proposition}[theorem]{Proposition}

\theoremstyle{definition}
\newtheorem{definition}[theorem]{Definition}
\newtheorem{s}[theorem]{}
\newtheorem{remark}[theorem]{Remark}
\newtheorem{example}[theorem]{Example}

\theoremstyle{remark}
\newtheorem*{claim*}{\it Claim}
\newtheorem*{note*}{\bf Note}

\begin{document}

\title[Graded components of local cohomology modules]{Graded components of local cohomology modules of invariant rings}
\address{Department of Mathematics, Indian Institute of Technology Bombay, Powai, Mumbai 400 076, India}
\author{Tony J. Puthenpurakal}
\email{\href{mailto:tputhen@math.iitb.ac.in}{tputhen@math.iitb.ac.in}}
\author{Sudeshna Roy}
\email{\href{mailto:sudeshnaroy@math.iitb.ac.in}{sudeshnaroy@math.iitb.ac.in}}
\date{\today}

\begin{abstract}
Let $A$ be a regular domain containing a field $K$ of characteristic zero, $G$ be a finite subgroup of the group of automorphisms of $A$ and $B=A^G$ be the ring of invariants of $G$. Let $S= A[X_1,\ldots, X_m]$ and $R= B[X_1, \ldots, X_m]$ be standard graded with $\deg A=0$, $\deg B=0$ and $\deg X_i=1$ for all $i$. Extend the action of $G$ on $A$ to $S$ by fixing $X_i$. Note $S^G=R$. Let $I$ be an arbitrary homogeneous ideal in $R$. The main goal of this paper is to establish a comparative study of graded components of local cohomology modules $H_I^i(R)$ that would be analogs to those proven in the paper \cite{TP2} for $H_J^i(S)$ where $J$ is an arbitrary homogeneous ideal in $S$. 
\end{abstract}
\maketitle
\section{introduction}
\s\label{sa}{\it Standard assumption:} Throughout this paper $A$ is a regular domain containing a field $K$ of characteristic zero, $G$ is a finite subgroup of the group of automorphisms of $A$ and $B=A^G$ is the ring of invariants of $G$. Let $S= A[X_1, \ldots, X_m]$ and $R= B[X_1, \ldots, X_m]$ be standard graded with $\deg A=0$, $\deg B=0$ and $\deg X_i=1$ for all $i$. Extend the action of $G$ on $A$ to $S$ by fixing $X_i$. Note $S^G=R$. Set $M=T(R)= \bigoplus_{n \in \Z} M_n$ where $T(-) = H^{i_1}_{I_1}(H^{i_2}_{I_2}(\cdots H^{i_r}_{I_r}(-) \cdots)$  for some homogeneous ideals $I_1, \ldots, I_r$ in $R$ and $i_1,\ldots,i_r \geq 0$. Set $N=T'(S)= \bigoplus_{n \in \Z} N_n$ where $T'(-)= H^{i_1}_{I_1S}(H^{i_2}_{I_2S}(\cdots H^{i_r}_{I_rS}(-) \cdots)$ where $T'(-)= H^{i_1}_{I_1S}(H^{i_2}_{I_2S}(\cdots H^{i_r}_{I_rS}(-) \cdots)$.

\begin{note*} 
Since $G$ is a finite group, any element $a \in A$ satisfies the polynomial $f_a(Z)= \prod_{g \in G}(Z-ga) \in B[Z]$ and hence is integral over $B$. So by \cite[Theorem 6.4.5]{BH} we have $B$ is Cohen-Macaulay (as $A$ is regular and hence Cohen-Macaulay). 
\end{note*}

In the paper \cite{TP2} we have seen that $T'(S)$ behaves nicely and has several good properties. In this paper we have proved the following analogous results for $T(R)$.

\vspace{0.15cm}
\noindent
{\bf I.} ({\it Bass numbers:})
The $j$-th Bass number of an $R$-module $E$ with respect to a prime ideal $P$ is defined as $\mu_j(P, E)= \dim_{k(P)} \Ext_{R_P}(k(P), E_P)$ where $k(P)$ is the residue field of $R_P$. Now we know that $\mu_j(P, E)$ is always a finite number (possibly zero) for all $j \geq 0$ if $E$ is a finitely generated $R$-module. But homogeneous components of $M=T(R)$ need not be finitely generated as $B$-modules. So $\mu_j(P, M_n)$ may not be a finite number. If $B_P$ is Gorenstein for some prime ideal $P$ of $B$ then we get the following result.

\begin{theorem}[with hypotheses as in \ref{sa}]\label{ibass}
Let $P$ be a prime ideal in $B$ such that $B_P$ is Gorenstein. Fix $j \geq 0$. Then EXACTLY one of the following holds:
\begin{enumerate}[\rm(i)]
\item $\mu_j(P, M_n)$ is infinite for all $n \in \Z$.
\item $\mu_j(P, M_n)$ is finite for all $n \in \Z$. In this case EXACTLY one of the following holds:
\begin{enumerate}[(a)]
\item $\mu_j(P, M_n)=0$ for all $n \in \Z$.
\item $\mu_j(P, M_n) \neq 0$ for all $n \in \Z$.
\item $\mu_j(P, M_n)\neq 0$ for all $n \geq 0$ and $\mu_j(P, M_n)= 0$ for all $n<0$.
\item $\mu_j(P, M_n)\neq 0$ for all $n \leq -m$ and $\mu_j(P, M_n)= 0$ for all $n > -m$.
\end{enumerate}
\end{enumerate}
\end{theorem}

In the following result $B$ is Cohen-Macaulay but not necessarily Gorenstein.
\begin{theorem}[with hypotheses as in \ref{sa}]
Assume $m=1$. Fix $j \geq 0$. Let $P$ be a prime ideal in $B$. Then $\mu_j(P, M_n)$ is finite for all $n \in \Z$. 
\end{theorem}

\noindent
{\bf II.} ({\it Growth of Bass numbers:})
Fix $j \geq 0$. Let $P$ be a prime ideal in $B$ such that $B_P$ is Gorenstein and $\mu_j(P, M_n)$ is finite for all $n \in \Z$. We now investigate the growth of the function $n \mapsto \mu_j(P, M_n)$ as $n \to \infty$ and as $n \to - \infty$.
\begin{theorem}[with hypothesis as in \ref{sa}]
Let $P$ be a prime ideal in $B$ such that $B_P$ is Gorenstein. Fix $j \geq 0$. Suppose $\mu_j(P, M_n)$ is finite for all $n \in \Z$. Then there exist polynomials $f_{M}^{j, P}(Z), g_M^{j, P}(Z)\in \Q[Z]$ of degree $\leq m-1$ such that \[f_{M}^{j, P}(n)=\mu_j(P, M_n) \text{ for all } n\ll0 \quad\text{AND}\quad g_{M}^{j, P}(n)=\mu_j(P, M_n) \text{ for all } n\gg0.\]
\end{theorem}
Fix $j \geq 0$. If $M_c=0$ for some $c$ then for any prime ideal $P$ in $B$ we get $\mu_j(P, M_c)=0$ is finite and hence by Theorem \ref{ibass} it follows that $\mu_j(P, M_n)$ is finite for all $n \in \Z$. For such cases we prove the following result.
\begin{theorem}[with hypothesis as in \ref{sa}]
Let $P$ be a prime ideal in $B$ such that $B_P$ is Gorenstein. Fix $j \geq 0$. Suppose $\mu_j(P, M_c)=0$ for some $c$ {\rm(} this holds if for instance $M_c=0$\rm{)}. Then 
\begin{align*}
&f_{M}^{j, P}(Z)=0  \quad \text{or} \quad \deg f_{M}^{j, P}(Z)= m-1,\\
&g_{M}^{j, P}(Z)=0 \quad \text{or} \quad \deg g_{M}^{j, P}(Z)= m-1.
\end{align*}
\end{theorem}

\noindent
{\bf III.} ({\it Dimension of Supports and injective dimension:})
The support of a $B$-module $V$ is defined as \[\Supp_B V=\{P~|~V_P \neq 0 \text{ and } P \text{ is a prime in } B\}.\] By $\dim_B V$ we mean the dimension of $\Supp_B V$ as a subspace of $\Spec(B)$. Let $\injdim_B V$ denotes the injective dimension of $V$. We show the following:
\begin{theorem}[with hypothesis as in \ref{sa}]\label{g}
If $B$ is Gorenstein then the following hold:
\begin{enumerate}[\rm (i)]
\item $\injdim M_c \leq \dim M_c$ for all $c \in \Z$.
\item $\injdim M_n= \injdim M_{-m}$ for all $n \leq -m$.
\item $\injdim M_n= \injdim M_0$ for all $n \geq 0$.
\item If $m \geq 2$ and $-m<r, s<0$ then 
\begin{enumerate}[\rm (a)]
\item $\injdim M_n= \injdim M_{s}$.
\item $\injdim M_n\leq \min\{ \injdim M_{-m}, \injdim M_0\}$.
\end{enumerate}
\end{enumerate}
\end{theorem}

We will also prove the following result.
\begin{theorem}[with hypotheses as in \ref{sa}]\label{ng}
Assume $m=1$. Let $P$ be a prime ideal in $B$ such that $B_P$ is not Gorenstein. Fix $n \in \Z$. Then EXACTLY one of the following holds:
\begin{enumerate}[\rm(i)]
\item $\mu_j(P, M_n)=0$ for all $j$.
\item there exists $c$ such that $\mu_j(P, M_n)=0$ for $j<c$ and $\mu_j(P, M_n)>0$ for all $j \geq c$.
\end{enumerate}
\end{theorem}

Let $M_n \neq 0$ and $\injdim_B M_n< \infty$ for some $n$. It should be noted that if $\mu_j(P, M_n) \neq 0$, then $B_P$ is Gorenstein.

\vspace{0.15cm}
\noindent
{\bf IV.} ({\it Associate primes:})
A prime ideal $P$ is associated to $E$ if there is some element $x$ of $E$ such that $\ann(x) =P$. The set of all such prime ideals is denoted by $\Ass(E)$. In this paper, we investigate finiteness and the asymptotic behavior of the set of primes associated to the $B$-module $M_n$. We prove the following: 

\begin{theorem}[with hypothesis as in \ref{sa}]
Further assume that either $A$ is local or a smooth affine algebra over a field $K$ of characteristic zero. Then $\bigcup_{n \in \Z}\Ass_B M_n$ is a finite set. 
	
Moreover, if $B$ is Gorenstein then
\begin{enumerate}[\rm(1)]
\item $\Ass_B M_n= \Ass_B M_{-m}$ for all $n \leq -m$.
\item $\Ass_B M_n= \Ass_B M_0$ for all $n \geq 0$.
\end{enumerate}
\end{theorem}

\noindent
{\bf V.} ({\it Infinite generation:})
Let $R=\bigoplus_{n \geq0}R_n$ be a positively graded ring and $M$ be a finitely generated graded $R$-module. Then by \cite[Theorem 16.1.5]{BS}, the $R_0$-module $H^i_{R_+}(M)_n$ is finitely generated  for all $i \in \N_0$ and $n \in \Z$.
Now by \cite[Theorem 1.7]{TP2}, we have one sufficient condition for infinite generation of a component of graded local cohomology module over $R$. In this paper, we give another one as follows. 
\begin{theorem}[with hypotheses as in \ref{sa}]
Let $J$ be a homogeneous ideal in $R$ such that $J \cap B \neq 0$. If $B$ is Gorenstein and $H_J^i(R)_c \neq 0$ then $H_J^i(R)_c$ is NOT finitely generated as a $B$-module.
\end{theorem}

We begin Section 2 with some basic properties of skew group rings and the action of $G$ on the graded components of $N$ that we need later on. In Section 3, we discuss for a fixed $j$, how Bass numbers $\mu_j(P, M_n)$ relate with each other for all $n$ when $P$ is a prime ideal in $B$ such that $B_P$ is Gorenstein. In Section 4, we study the behavior of the function $n \mapsto \mu_j(P, M_n)$ as $n \to \infty$ and as $n \to - \infty$. In Section 5, we talk about Bass numbers of $M_n$ when $B_P$ is NOT Gorenstein and $m=1$. In Sections 6 and 7, we study finiteness of injective dimensions and associated primes respectively and establish relations between injective dimensions and associated primes of graded components under certain conditions. Finally in Section 8, we give a sufficient condition under which $M_n$ is not finitely generated as a $B$-module.

\section{Skew Group Rings and Graded Local Cohomology}
\s {\it Recall:} Let $A$ be a ring (not necessarily commutative) and $G$ is a finite subgroup of $Aut(A)$ and $|G|$ is invertible in $A$. 

The skew-group ring of $A$ (\wrt~ $G$) is
\[A*G = \{ \sum_{\sigma \in G} a_\sigma \sigma \mid a_{\sigma} \in A \ \text{for all} \ \sigma \},\]
with multiplication defined as
\[
(a_{\sigma}\sigma)(a_\tau \tau) = a_{\sigma}\sigma(a_{\tau})\sigma \tau.
\]  
An $A*G$ module $M$ is an $A$-module on which $G$ acts such that for all $\sigma \in G$,
\[
\sigma(am) = \sigma(a)\sigma(m) \quad \text{for all} \ a \in A \ \text{and} \ m \in M.
\]

\begin{definition}
Let $M$ be an $A*G$-module. Then
\[M^G = \{ m \in M \mid \sigma(m) = m \ \text{for all} \ \sigma \in G \}.\]
\end{definition}
Let us set $A^G$ to be the ring of invariants of $G$. Let $M, N$ be $A*G$-modules. It can be easily checked that 
\begin{enumerate}[\rm (1)]
\item $M^G$ is an $A^G$-module.
\item If $u \colon M \rt N$ is $A*G$-linear then $u(M^G) \subseteq N^G$ and the restriction map $\widetilde{u} \colon M^G \rt N^G$ is $A^G$-linear. Thus we have a functor \[(-)^G \colon Mod(A*G) \rt Mod(A^G).\]
\item $(-)^G = \Hom_{A*G}(A,-)$ and hence it is left exact.
\end{enumerate}

\s{\it Reynolds operator:}
For any $A*G$-module $M$ define the Reynolds operator 
\begin{align*}
\rho^M \colon M &\rt M^G \\
m &\mapsto \frac{1}{|G|}\sum_{\sigma \in G} \sigma m.
\end{align*}
Then clearly $\rho^M$ is $A^G$-linear and $\rho^M(m)=m$ for all $m \in M^G$.

Since $|G|$ is invertible in $A$ so Reynold operator exists. Therefore by \cite[Lemma 3.5]{TP1} it follows that $(-)^G$ is an exact functor.

\s Let $A$ be a commutative Noetherian ring and $G$ be a finite subgroup of the group of automorphisms of $A$ with $|G|$ invertible in $A$. Let $B=A^G$ be the ring of invariants of $G$. Let $S= A[X_1, \ldots, X_m]$ and $R=B[X_1, \ldots, X_m]$. Let $G$ acts on $A$ and fixes $X_i$. Then clearly $R=S^G$. Set $M=T(R)= \bigoplus_{n \in \Z} M_n$ where $T(-) = H^{i_1}_{I_1}(H^{i_2}_{I_2}(\cdots H^{i_r}_{I_r}(-) \cdots)$  for some ideals $I_1, \ldots, I_r$ in $R$ and $i_1,\ldots,i_r \geq 0$. Set $N=T'(S)= H^{i_1}_{I_1S}(H^{i_2}_{I_2S}(\cdots H^{i_r}_{I_rS}(S) \cdots)= \bigoplus_{n \in \Z} N_n$. Then by \cite[Corollary 4.3]{TP1} we get that $T'(S)$ is a $S*G$-module and $T'(S)^G= T(R)$.

All of our results depend on the following statement.
\begin{proposition}\label{inva}
$N_n$ is an $A*G$-module and $N_n^G=M_n$ for all $n$.
\end{proposition}
\begin{proof}
As $N$ is a graded $S$-module we have $N_n$ is a $S_0=A$-module. We also have $N=T'(S)$ is a $S*G$-module. Now $A \subseteq S$ is a sub-ring. So $N$ is an $A*G$-module. Notice $G$ acts linearly on $S$ (fixes $X_i$). Thus $\sigma(N_n) \subseteq N_n$ for all $\sigma \in G$. Moreover, $N_n \subseteq_A N$ and hence $N_n$ is an $A*G$-module. Clearly $\rho^N: N \to N^G$ is a degree zero $R$-module homomorphism. It follows that $N_n^G=M_n$.
\end{proof}

%We have the following observation.
%\begin{lemma}
%$B$ contains an infinite field.
%\end{lemma}
%\begin{proof}
%Clearly for all $\sigma \in G, ~ \sigma(1)=1$. So for all $r \in \Z$, $\sigma(r)= r \sigma(1)=r$ and $\sigma(-r)= - \sigma(r)=-r$. Since $s. (r/s)=r$ so we have $\sigma(s. (r/s))= \sigma(r)=r$, i.e., $\sigma(s)= \sigma(r/s)=r$. Therefore $s.\sigma(r/s)=r$ and hence $\sigma(r/s)=r/s$ or all $\sigma \in G$. Thus $\Q \subseteq B$.
%\end{proof}

\section{Bass numbers}
\noindent
\s\label{bs}{\it Setup:} Let $A$ be a regular domain containing a field of characteristic zero. Let $G$ be a finite subgroup of the group of automorphisms of $A$. Let $B=A^G$ be the ring of invariants of $G$. Let $S= A[X_1,\ldots, X_m]$ and $R= B[X_1,\ldots, X_m]$ be standard graded with $\deg A=0$, $\deg B=0$ and $\deg X_i=1$ for all $i$. 

Set $M=T(R)= \bigoplus_{n \in \Z} M_n$ where $T(-) = H^{i_1}_{I_1}(H^{i_2}_{I_2}(\cdots H^{i_r}_{I_r}(-) \cdots)$  for some ideals $I_1, \ldots, I_r$ in $R$ and $i_1,\ldots,i_r \geq 0$ and $T'(S)= \bigoplus_{n \in \Z}N_n$ where $T'(-)= H^{i_1}_{I_1S}(H^{i_2}_{I_2S}(\cdots H^{i_r}_{I_rS}(-) \cdots)$.

Note that $\mu_j(P, M)=\mu_j(PR_P, M_P)$. So by \cite[Lemma 4.1, Lemma 4.4]{TP1}, it is enough to prove any result of Bass numbers only for maximal ideals after localizing. In this section we will take $(B, \m)$ is a local ring with $\hgt \m = d$. Set $E = E_{B}(B/\m)$ and  $l = B/\m$. Let $\n_1, \n_2, \ldots, \n_r$ be {\it all} the maximal ideals of $A$ lying over $\m$. Since $A$ is a normal domain and $N_n$ is an $A*G$-module (by Proposition \ref{inva}) so by \cite[Theorem 6.1]{TP1} we get $H_{\m A}^j(N_n)= \bigoplus_{l=1}^r H_{\n_l}^j(N_n)$.

\begin{lemma}\label{inf}
Let $\hgt P=g$. Then \[\left(H_P^j(N_n^G)\right)_P = H_{PB_P}^g(B_P)^{s_j(n)} \quad \text{for some } s_j(n) \geq0.\] Here $s_j$ is some cardinal {\rm(}possibly infinite{\rm)}.
\end{lemma}
\begin{proof}
It suffices to prove this result only for maximal ideal considering $(B, \m)$ is a local ring. Clearly $\dim B =g$. We have $H_{\m A}^j(N_n)= \bigoplus_{l=1}^r H_{\n_l}^j(N_n)$. By \cite[Proposition 7.3]{TP1} $H_{\n_l}^j(N_n)= (H_{\n_l}^j(N_n))_{\n_l}$. Now from the proof of \cite[Proposition 9.4]{TP2} we have $H_{\n_l}^j(N_n)$ is an injective module and $H_{\n_l}^j(N_n)\cong E_A(A/\n_l)^{\alpha_{lj}(n)}$ where $\alpha_{lj}(n)$ is some cardinal possibly infinite. Thus $H_{\m A}^j(N_n)= \bigoplus_{l=1}^r E_A(A/\n_l)^{\alpha_{lj}(n)}$. Since $A_{n_l}$ is a Gorenstein local ring so we have \[H_{\n_l A}^g(A) = (H_{\n_l A}^g(A))_{\n_l}= H_{\n_l A_{\n_l}}^g(A_{\n_l})\cong E_A(A/\n_l).\] Thus we get $H_{\m A}^j(N_n)= \bigoplus_{l=1}^r H_{\n_l A}^g(A)^{\alpha_{lj}(n)}$. Moreover, by \cite[Theorem 6.1]{TP1} we have $\sigma_l^k(H_{\n_l}^j(N_n)) = H_{\n_k}^j(N_n)$, where $\sigma_l^k(\n_l)= \n_k$ for all $l, k$. Therefore $\alpha_{1j}(n)= \cdots= \alpha_{rj}(n)= s_j(n)$ (say). Thus we have 
\begin{equation}\label{be1}
H_{\m A}^j(N_n)= \left(\bigoplus_{l=1}^r H_{\n_l A}^g(A)\right)^{s_j(n)} \cong H_{\m A}^g(A)^{s_j(n)}
\end{equation}
of $A*G$-modules. Again by Proposition \ref{inva} we have $M_n= N_n^G$. Now applying $(-)^G$ on both sides of \eqref{be1} and by \cite[Corollary 4.3(2)]{TP1}, we get $H^j_{\m}(N_n^G) \cong H^g_{\m}(A^G)^{s_j(n)}= H^g_{\m}(B)^{s_j(n)}.$
\end{proof}
 As an application of the above Lemma we prove the following.
\begin{proposition}\label{inj}
Let $P$ be a prime ideal in $B$ such that $B_P$ is Gorenstein. Set $V= M_n$. Then $(H_P^j(V))_P$ is injective $B$-module for all $j \geq 0$.
\end{proposition}
\begin{proof}
Since for any prime ideal $P$ in $B$, $(B_P, PB_P)$ is a Gorenstein local ring so we have $H_{PB_P}^g(B_P) \cong E_B(B/P)$ where $\hgt P=g$. Thus by Lemma \ref{inf} we get $(H_P^j(V))_P \cong E_B(B/P)^{s_j(n)}$ for some $s_j(n) \geq 0$ and hence it is injective.
\end{proof}

We need the following lemma from \cite[1.4]{Lyu-1}.
\begin{lemma}\label{ilyu}
Let $A$ be a Noetherian ring and let $M$ be an $A$-module {\rm(}$M$ need not be finitely generated{\rm)}. Let $P$ be a prime ideal in $A$. If $(H_P^j(M))_P$ is injective $A$-module for all $j \geq 0$ then $\mu_j(P, M)= \mu_0(P, H_P^j(M))$.
\end{lemma}

As an immediate consequence we get the following.
\begin{lemma}\label{shiftbass}
Let $P$ be a prime ideal in $B$ such that $B_P$ is Gorenstein. Then $\mu_j(P, M_n)= \mu_0(P, H_P^j(M_n))$ for all $j \geq 0$.
\end{lemma}
\begin{proof}
The result follows from Proposition \ref{inj} and Lemma \ref{ilyu}.
\end{proof}

We are now ready to prove the main result of this section. This shows that what the first author observed in \cite{TP2} also holds in our case.
\begin{theorem}\label{bass}
Let $P$ be a prime ideal in $B$ such that $B_P$ is Gorenstein. Fix $j \geq 0$. Then EXACTLY one of the following holds:
\begin{enumerate}[\rm(i)]
\item $\mu_j(P, M_n)$ is infinite for all $n \in \Z$.
\item $\mu_j(P, M_n)$ is finite for all $n \in \Z$. In this case EXACTLY one of the following holds:
\begin{enumerate}[(a)]
\item $\mu_j(P, M_n)=0$ for all $n \in \Z$.
\item $\mu_j(P, M_n) \neq 0$ for all $n \in \Z$.
\item $\mu_j(P, M_n)\neq 0$ for all $n \geq 0$ and $\mu_j(P, M_n)= 0$ for all $n<0$.
\item $\mu_j(P, M_n)\neq 0$ for all $n \leq -m$ and $\mu_j(P, M_n)= 0$ for all $n > -m$.
\end{enumerate}
\end{enumerate}
\end{theorem}

\begin{proof}
If $M_n =0$ then $\mu(P, M_n)< \infty$. So without loss of generality we can take $M_n \neq 0$. It is enough to prove this result only for maximal ideal considering $(B, \m)$ is Gorenstein local. Then from the proof of \cite[Theorem 6.1]{TP1} we get $H_{\m A}^j(N_n) = \bigoplus_{l=1}^r H_{\n_l}^j(N_n)= \bigoplus_{l=1}^r E_A(A/\n_l)^{s_j(n)}$ for all $j \geq 0$, where $\n_1, \ldots, \n_r$ are all the maximal ideals of $A$ lying over $\m$. 
\begin{claim*}
$\mu_j(\m, M_n)$ is finite if and only if $s_j(n)$ is finite.
\end{claim*}
By Lemma \ref{shiftbass} we have $\mu_j(\m, M_n)= \mu_0(\m, H_{\m}^j(M_n))$ for all $j \geq 0$. Now by Lemma \ref{inf} we get $H_{\m}^j(M_n) = H_{\m}^g(B)^{s_j(n)}$ where $\dim~B= g$. Furthermore, as $(B, \m)$ is a Gorenstein local ring so $H_{\m}^g(B) \cong E$. Therefore $\mu_0(\m, H_{\m}^j(M_n))= s_j(n)$ and hence $\mu_j(\m, M_n)=s_j(n)$. The claim follows.
	
If $\mu_j(\m, M_n)$ is finite for some $n$ then by the above claim we get $s_j(n)$ is finite. Therefore $\mu_0(\n_l, H_{\m A}^j(N_n))$ is finite for any $l$. Fix $l$. Note that $H_{\m A}^j(N_n)= (H_{\m S}^j(N))_n= (H_{\m S}^j(T'(S)))_n$ and $H_{\m S}^j(T'(-))$ is a Lyubeznik functor on $^*\Mod(S)$. Then by \cite[Theorem 9.2]{TP2} we have $\mu_0(\n_l, H_{\m A}^j(N_r))= s_j(r)$ is finite for all $r \in \Z$ and satisfies one of $(a), ~(b), ~(c), ~(d)$. Therefore by the above claim $\mu(\m, M_r)= s_j(r)$ is finite for all $r$ and satisfies one of $(a), ~(b), ~(c), ~(d)$.
\end{proof}

\section{Growth of Bass numbers}
In this section we will study the behavior of the function $n \mapsto \mu_j(P, M_n)$ as $n \to \infty$ and as $n \to - \infty$.
\begin{theorem}[with hypothesis as in \ref{bs}]\label{growth}
Let $P$ be a prime ideal in $B$ such that $B_P$ is Gorenstein. Fix $j \geq 0$. Suppose $\mu_j(P, M_n)$ is finite for all $n \in \Z$. Then there exist polynomials $f_{M}^{j, P}(Z), g_M^{j, P}(Z)\in \Q[Z]$ of degree $\leq m-1$ such that \[f_{M}^{j, P}(n)=\mu_j(P, M_n) \text{ for all } n\ll0 \quad\text{AND}\quad g_{M}^{j, P}(n)=\mu_j(P, M_n) \text{ for all } n\gg0.\]
\end{theorem}

\begin{proof}
As in Section 3, it is enough to prove this result only for maximal ideal considering $(B, \m)$ is local. By the claim in the proof of Theorem \ref{bass} we have $\mu_j(\m, M_n)= s_j(n)= \mu_0(\n_l, H_{\m A}^j(N_n))$ for any $1 \leq l \leq r$ where $\n_1, \cdots, \n_r$ are all the maximal ideals of $A$ lying over $\m$. Set $V= H_{\m S}^j(N)$. Clearly $V_n= H_{\m A}^j(N_n)$. Fix $l$. Since $\mu_j(\m, M_n)$ is finite for all $n \in \Z$ so we get $\mu_0(\n_l, H_{\m A}^j(N_n))$ is finite for all $n \in \Z$. Therefore by \cite[Theorem 1.11]{TP2} there exist polynomials $f_{V}^{0, \n_l}(Z), g_V^{0, \n_l}(Z)\in \Q[Z]$ of degree $\leq m-1$ such that $f_{V}^{0, \n_l}(n)=\mu_0(\n_l, H_{\m A}^j(N_n))$ for all $n\ll0$ and $g_{V}^{0, \n_l}(n)=\mu_0(\n_l, H_{\m A}^j(N_n))$ for all $n\gg0$. Take $f_M^{j, \m}(Z)= f_{V}^{0, \n_l}(Z)$ and $g_M^{j, \m}(Z)= g_V^{0, \n_l}(Z)$. The result follows.
\end{proof}

The following result gives some properties of the polynomials appeared in the foregoing Theorem.
\begin{theorem}[with hypothesis as in \ref{bs}]\label{growth1}
Let $P$ be a prime ideal in $B$ such that $B_P$ is Gorenstein. Fix $j \geq 0$. Suppose $\mu_j(P, M_c)=0$ for some $c$ {\rm(} this holds if for instance $M_c=0$\rm{)}. Then 
\begin{align*}
&f_{M}^{j, P}(Z)=0  \quad \text{or} \quad \deg f_{M}^{j, P}(Z)= m-1,\\
&g_{M}^{j, P}(Z)=0 \quad \text{or} \quad \deg g_{M}^{j, P}(Z)= m-1.
\end{align*}
\end{theorem}

\begin{proof}
We have $\mu_j(\m, M_c)= 0 = \mu_0(\n_l, H_{\m A}^j(N_c))$. Set $V= H_{\m S}^j(N)$. Clearly $V_n= H_{\m A}^j(N_n)$. Fix $l$. As $f_M^{j, \m}(Z)= f_{V}^{0, \n_l}(Z)$ and $g_M^{j, \m}(Z)= g_V^{0, \n_l}(Z)$ so by \cite[Theorem 1.12]{TP2} it follows that $f_{M}^{j, \m}(Z)=0 \text{ or } \deg f_{M}^{j, \m}(Z)= m-1$ and $g_{M}^{j, \m}(Z)=0 \text{ or } \deg g_{M}^{j, \m}(Z)= m-1$. 
\end{proof}

\section{Bass numbers when \texorpdfstring{$B_P$}{Bp} is NOT Gorenstein and \texorpdfstring{$m=1$}{m=1}}

We now concentrate on the case when $m=1$. The following result gives us a sufficient condition under which for any fixed $j$ and prime ideal $P$ in $B$ the Bass number $\mu_j(P, M_n)$ is finite for all $n \in \Z$. Recall that under our assumption $B$ is always Cohen-Macaulay.

\begin{theorem}[with standard assumption \ref{sa}]\label{bass1}
Assume $m=1$. Fix $j \geq 0$. Let $B$ be \CM but not necessarily Gorenstein and $P$ be a prime ideal in $B$. Then $\mu_j(P, M_n)$ is finite for all $n \in \Z$. 
\end{theorem}

\begin{proof}
It is enough to prove this result only for maximal ideal considering $(B, \m)$ is local with $\dim B=d$. We have $H_{\m A}^j(N_n)= \bigoplus_{i=1}^r E_A(A/\n_i)^{s_j(n)}$ where $\n_1, \ldots, \n_r$ are all the maximal ideals of $A$ lying over $\m$. Since $H_{\m A}^j(N_n)= (H_{\m S}^j(N))_n$ so by \cite[Theorem 1.9]{TP2} we get $\mu_0(\n_l, H_{\m A}^j(N_n))= s_j(n)$ is finite for all $n\in \Z$. Thus if $B$ is Gorenstein then by the claim of Theorem \ref{bass} we are done.

Otherwise let $\GG$ be a minimal injective resolution of $M_n$ and
\[
\GG^j = \widetilde{\GG^j} \oplus E^{r_j} \quad \text{with} \ \m \notin \Ass(\widetilde{\GG^j}).
\]
Now $\mu_j(\m , M_n) = \dim_l \Ext_{B}^j(l, M_n)$ for $j \geq 0$ and $\Ext_{B}^j(l, M_n) = \Hom_{B}(l, \GG^j)= \Hom_{B}(l, \widetilde{\GG^j} \oplus E^{r_j})= \Hom_{B}(l, \widetilde{\GG^j}) \oplus  \Hom_{B}(l, E^{r_j})= \Hom_{B}(l, E^{r_j})= E_l(l)^{r_j}= l^{r_j}$ and hence $\mu_j(\m, M_n)= r_j$. So we have to prove that $r_j$ is finite for all $j \geq 0$.

Set $\EE = \Gamma_\m(\GG)$. Since $\m \notin \Ass(\widetilde{\GG^j})$ so we have $\Gamma_\m(\widetilde\GG^j)= 0$. Now it is well known that $\Gamma_\m(E^{r_j})= E^{r_j}$. Thus $\EE^j= E^{r_j}$ for all $j \geq 0$. Furthermore by Lemma \ref{inf} we have
\[
H^j(\EE) = H^j_\m(M_n) \cong H^d_\m(B)^{s_j(n)} \quad \text{for some finite} \ s_j(n) \geq 0.
\] 
Note that finiteness of $s_j(n)$ follows from the first part of this proof.
\begin{claim*}
 $\dim_l \Hom_{B}(l, H^j(\EE))$ is finite for all $j \geq 0$.
\end{claim*}
Let $\widehat{B}$ be the completion of $B$ at $\m$. Then $(\widehat{B}, \widehat{\m})$ is a local ring of dimension $d$ with maximal ideal $\widehat{\m}= \m \widehat{B}$ and $E = E_{\widehat{B}}(\widehat{B}/\m \widehat{B}) = E_{\widehat{B}}(l)$. By \cite[3.5.4(d)]{BH} we get $H^d_\m(B) = H^d_{\m \widehat{B}}(\widehat{B})$ and by \cite[3.5.4(a)]{BH} we get $H_\m^d(B)$ is Artinian. Therefore $H^d_\m(B) = H^d_{\m \widehat{B}}(\widehat{B}) \subseteq E^t$ where $t= \dim_l \soc(H_\m^d(B)) < \infty$. The finiteness comes as $\soc (H_\m^d(B)) \subset H_\m^d(B)$ is Artinian (and hence is a finite dimensional $l$ vector space). Hence 
\begin{align*}
\dim_l \Hom_{B}(l, H^j(\EE)) &= \dim_l \Hom_{B}(l, H_\m^d(B))^{s_j(n)} \\
&\leq \dim_l \Hom_{B}(l, E^t)^{s_j(n)}\\
&= \dim_l l^{ts_j(n)} = ts_j(n)
\end{align*} 
is finite and the claim follows.

\vspace{0.25cm}
We now prove by induction that $r_j$ is finite for all $j \geq 0$.

Since $\GG$ is a minimal injective resolution of $M_n$ so we have
$\Hom_{B}(l, \GG^j) \rt \Hom_{B}(l, \GG^{j+1})$ is a zero map for all $j$. As $\Hom_{B}(l, \widetilde{\GG^j})= 0$ for all $j \geq 0$ so we get $\Hom_{B}(l, E^{r_j}) \rt \Hom_{B}(l, E^{r_{j+1}})$ is a zero map for all $j\geq 0$.

For $j = 0$, we have an exact sequence
\[
0 \rt H^0(\EE) \rt E^{r_0} \rt E^{r_1}.
\]
Applying $\Hom_{B}(l,-)$ we get another exact sequence
\[
0 \rt \Hom_{B}(l, H^0(\EE))  \rt l^{r_0} \xrightarrow{0} l^{r_1}.
\]
Hence $l^{r_0} \cong \Hom_{B}(l, H^0(\EE))$ and $r_0=  \dim_l \Hom_{B}(l, H^0(\EE))$ is finite by our claim.

Now let us assume that $r_0,\ldots,r_{m-1}$ are finite and consider the following part of $\EE$;
\[
E^{r_{m-1}} \xrightarrow{d^{m-1}} E^{r_m} \xrightarrow{d^{m}} E^{r_{m+1}}.
\]
Set $Z^m = \ker d^m$ and $B^m = \im d^{m-1}$. Then we have an exact sequence
\[
0 \rt B^m \rt Z^m \rt H^m(\EE) \rt 0.
\]
Since $E$ is Artin and by induction hypothesis $r_{m-1} < \infty$ so we get $E^{r_{m-1}}$ is Artin. Therefore $B^m$ is also Artin and hence is a $\m$-torsion module. Therefore $B^m \subseteq E^s$ for some $s > 0$ (can take $s= \dim_l \soc(B_m)$). Thus $\dim_l (\Hom_{B}(l, B^m)) \leq \dim_l \Hom_{B}(l, E^s)= s < \infty$. 
Since 
\[
0 \rt \Hom_{B}(l, B^m) \rt \Hom_{B}(l, Z^m) \rt \Hom_{B}(l, H^m(\EE))
\]
is an exact sequence so by our claim it follows that $\Hom_{B}(l, Z^m)$ is a finite dimensional $l$-vector space. 

Again we have an exact sequence $0 \rt Z^m \rt E^{r_m} \xrightarrow{d^m} E^{r_{m+1}}$. Applying $\Hom_{B}(l,-)$, we get
\[
0 \rt \Hom_{B}(l, Z^m) \rt \Hom_{B}(l, E^{r_m}) \xrightarrow{0} \Hom_{B}(l, E^{r_{m+1}}).
\]
Therefore $r_m= \dim_l \Hom_{B}(l, E^{r_m})= \dim_l \Hom_{B}(l, Z^m)$ is finite and the result follows.
\end{proof}

\begin{remark}
In the case $m \geq 2$, the above prove goes if $\mu_i(P, M_n) < \infty$ for sone $i$ and $\mu_j(P, M_n) < \infty$ for all $j \leq i$. But we don't know that whether $\mu_i(P, M_n)< \infty$ for some $i$ implies $\mu_j(P, M_n)< \infty$ for all $j<i$ or not.
\end{remark}

The next result gives us a sufficient condition when $\injdim M_n$ is infinite if $0 \neq M_n$ is not injective.

\begin{theorem}[with standard assumption \ref{sa}]\label{gbass1}
Assume $m=1$ and $B_P$ is not Gorenstein for any prime ideal $P$ of $B$. Fix $n \in \Z$. Then EXACTLY one of the following holds:
\begin{enumerate}[\rm(i)]
\item $\mu_j(P, M_n)=0$ for all $j$.
\item there exists $c$ such that $\mu_j(P, M_n)=0$ for $j<c$ and $\mu_j(P, M_n)>0$ for all $j \geq c$.
\end{enumerate}
\end{theorem}

\begin{proof}
Fix $n$. Let $\GG$ be a minimal injective resolution of $M_n$ and
\[
\GG^j = \widetilde{\GG^j} \oplus E^{r_j} \quad \text{with} \ \m \notin \Ass(\widetilde{\GG^j}).
\]
Now $\mu_j(\m , M_n) = \dim_l \Ext_{B}^j(l, M_n)$ for $j \geq 0$ and $\Ext_{B}^j(l, M_n) = \Hom_{B}(l, \GG^j)= \Hom_{B}(l, E^{r_j})= l^{r_j}$ and hence $\mu_j(\m, M_n)= r_j$. So by Theorem \ref{bass1} we have $r_j$ is finite for all $j \geq 0$. Let\[c= \min\{j~|~r_j>0\}.\] We have to prove that $r_j>0$ for all $j \geq c$.	
	
Set $\EE = \Gamma_\m(\GG)$. Since $\m \notin \Ass(\widetilde{\GG^j})$ so we have $\Gamma_\m(\widetilde{\GG^j})= 0$. Now it is well known that $\Gamma_\m(E^{r_j})= E^{r_j}$. Thus $\EE^j= E^{r_j}$ for all $j \geq 0$. Furthermore by Lemma \ref{inf} we have
\[
H^j(\EE) = H^j_\m(M_n) \cong H^d_\m(B)^{s_j(n)} \quad \text{for some } s_j(n) \geq 0,
\]
where $\dim B=d$. By the first part of this proof of Theorem \ref{bass1} we have $s_j(n)$ is finite for all $n$. Let $T$ be the completion of $B$ at $\m$. Then $(T, \n)$ is a local ring of dimension $d$ with maximal ideal $\n= \m T$ and $E = E_{T}(T/\n T) = E_{T}(l)$. By \cite[3.5.4]{BH} we have $H^d_\m(B) = H^d_{\n}(T)$ and $H_\m^d(B)$ is Artinian. 
	
Set $Z^j= \ker d^j$ and $B^j= \im d^{j-1}$. Let $(-)^\vee$ be the Matlis dual of $T$. Now we prove the following assertions by induction on $j \geq c$:
\begin{enumerate}[\rm(i)]
\item $Z^j \neq0$;
\item $\injdim Z^j = \infty$;
\item $(Z^j)^\vee$ is a non-free maximal Cohen-Macaulay $T$-module;
\item $B^{j+1} \neq 0$;
\item $\injdim B^{j+1} = \infty$;
\item $(B^{j+1})^\vee$ is a non-free maximal Cohen-Macaulay $T$-module;
\end{enumerate}
Although $(i)$ will prove our assertion we will prove all the above assertion together for $j \geq c$.
	
Since $\GG$ is a minimal injective resolution of $M_n$ so we have $\Hom_{B}(l, \GG^j) \rt \Hom_{B}(l, \GG^{j+1})$ is a zero map for all $j$. As $\Hom_{B}(l, \widetilde{\GG^j})= 0$ for all $j \geq 0$ so we get $\Hom_{B}(l, E^{r_j}) \rt \Hom_{B}(l, E^{r_{j+1}})$ is a zero map for all $j\geq 0$. Now we have an exact sequence $0 \to Z^j \to E ^{r_j} \overset{d^j}{\rt} E^{r_{j+1}}$. Applying $\Hom_B(l, -)$ we get another exact sequence \[0 \to \Hom_{B}(l, Z^j)\rt \Hom_{B}(l, E^{r_j})\overset{0}{\rt} \Hom_{B}(l, E^{r_{j+1}}) .\] Thus $\dim_l \Hom_{B}(l, Z^j)= r_j$. Therefore $Z^c \neq 0$. As $Z^j \subseteq E^{r_j}$ and $r_j=0$ for all $j<c$ it follows that $Z^j =0$ for all $j<c$. Since $B^{c} \cong E^{r_{c-1}}/Z^{c-1}=0$ so we have $Z^c= H^c(\EE)= H_{\m}^d(B)^{s_c(n)}$ for some finite $s_c(n) \geq 0$. As $B$ is not Gorenstein so $T$ is not Gorenstein and hence $H^d_{\n}(T)$ is not injective. Moreover, $H_{\n}^d(T)^\vee= \omega$ the canonical module of $T$. If $\omega$ is free then  $\projdim \omega=0$. But in that case as $\omega^\vee=H_{\n}^d(T)$ it follows that $\injdim H_{\n}^d(T)=0$, a contradiction. Therefore $\omega$ is a non-free maximal Cohen-Macaulay $T$-module. It follows that $(Z^c)^\vee$ is a non-free maximal Cohen-Macaulay $T$-module. Therefore $Z^c$ has infinite injective dimension. As otherwise $\injdim Z^c<\infty$ implies $\projdim (Z^c)^\vee<\infty$ and hence by {\it Auslander-Buchsbaum formula} we get $\projdim (Z^c)^\vee=0$, i.e., $(Z^c)^\vee$ is free (as $T$ is local), a contradiction.
	
We have an exact sequence $0 \to Z^c \to E^{r_c} \to B^{c+1} \to 0$. As $\injdim Z^c= \infty$ it follows that $B^{c+1} \neq 0$ and has infinite injective dimension. By taking $(-)^\vee$ we get an exact sequence \[0 \to (B^{c+1})^\vee \to T^{r_c} \to (Z^c)^\vee \to 0.\] It follows that $(B^{c+1})^\vee$ is a non-free maximal Cohen-Macaulay $T$-module.
	
We now assume the result is true for $j=m$ and prove it for $j=m+1$. We have an exact sequence \[0 \to B^{m+1} \to Z^{m+1} \to H^{m+1}(\EE) \to 0.\] By induction hypothesis as $B^{m+1} \neq 0$ and satisfies $(v)$ and $(vi)$. It follows that $Z^{m+1}\neq 0$. If $H^{m+1}(\EE)= 0$ then clearly $Z^{m+1} \cong B^{m+1}$ satisfies $(ii)$ and $(iii)$. If $H^{m+1}(\EE) \neq 0$ then taking Matlis-duals we get an exact sequence \[0 \to \omega^{s_{m+1}(n)} \to (Z^{m+1})^\vee \to (B^{m+1})^\vee \to 0.\] Now for any maximal Cohen-Macaulay $T$-module $N$ we have $\Ext^1_{T}(N, \omega)= 0$. In particular, $\Ext^1_{T}((B^{m+1})^\vee, \omega^{s_{m+1}(n)})= 0$. Therefore \[(Z^{m+1})^\vee \cong \omega^{s_{m+1}(n)} \oplus (B^{m+1})^\vee.\]Thus $(Z^{m+1})^\vee$ is a non-free maximal Cohen-Macaulay $T$-module. Since $r_i$ is finite so we have $Z^{i} \subseteq E^{r_i}$ is Artinian and hence $B^{i} \subseteq Z^{i}$ is Artinian for all $i$. Therefore by taking duals again we get that\[Z^{m+1} \cong H_{\n}^d(T)^{s_{m+1}(n)} \oplus B^{m+1};\] has infinite injective dimension. 
	
Again we have an exact sequence \[0 \to Z^{m+1} \to E^{r_{m+1}} \to B^{m+2} \to 0.\] As $\injdim Z^{m+1}= \infty$ it follows that $B^{m+2} \neq 0$ and has infinite injective dimension. Taking $(-)^\vee$ we get an exact sequence \[0 \to (B^{m+2})^\vee \to T^{r_{m+1}} \to (Z^{m+1})^\vee \to 0.\] It follows that $(B^{m+2})^\vee$ is a non-free maximal Cohen-Macaulay $T$-module and satisfies $({\rm v})$ and $({\rm vi})$.
\end{proof}

\section{Dimension of Support and injective dimension}

We begin with the following relation which shows that $\injdim M_n$ is finite for any $n \in \Z$ if $B$ is Gorenstein.
\begin{lemma}[with standard assumption \ref{sa}]\label{id}
Let $c \in \Z$. If $B$ is Gorenstein then \[\injdim M_c \leq \dim M_c.\]
\end{lemma}

\begin{proof}
Let $P$ be a prime ideal in $B$. Now by Proposition \ref{inj} along with \cite[1.4]{Lyu-1} we get \[\mu_j(P, M_c)= \mu_0(P, H^j_P(M_c)).\] Moreover, by Grothendieck's Vanishing Theorem $H^j_P(M_c)= 0$ for all $j> \dim M_c$. So $\mu_j(P, M_c)=0$ for all $j> \dim M_c$. 
\end{proof}

The following example shows that Lemma \ref{id} does not hold true if $B$ is not Gorenstein.
\begin{example}
Let $A= \CC[[Y_1, \ldots, Y_n]]$ and $G \subseteq Gl_n(\CC)$ acting linearly with $A^G$ NOT Gorenstein. Let $\m$ and $\m^G$ be maximal ideals of $A$ and $A^G=B$ respectively. As $A^G$ is NOT Gorenstein we have $\injdim H_{\m^G}^n(B)= \infty$. Set $S=A[X_1, \ldots, X_m]$ and $R=B[X_1, \ldots, X_m]$. Set $M=H_{\m^GR}^n(R)= H_{\m^G}^n(B) \otimes_B R$. It follows that $\injdim_B M_0= \infty$.
\end{example}

We now establish the following under the extra hypothesis that $B$ is Gorenstein.
\begin{theorem}[with standard assumption \ref{sa}]\label{idsd}
If $B$ is Gorenstein then the following hold:
\begin{enumerate}[\rm (i)]
\item $\injdim M_c \leq \dim M_c$ for all $c \in \Z$.
\item $\injdim M_n= \injdim M_{-m}$ for all $n \leq -m$.
\item $\injdim M_n= \injdim M_0$ for all $n \geq 0$.
\item If $m \geq 2$ and $-m<r, s<0$ then 
\begin{enumerate}[\rm (a)]
\item $\injdim M_r= \injdim M_{s}$.
\item $\injdim M_r\leq \min\{ \injdim M_{-m}, \injdim M_0\}$.
\end{enumerate}
\end{enumerate}
\end{theorem}

\begin{proof}
Let $B$ be Gorenstein and $P$ be a prime ideal in $B$. 

(i) This follows from Lemma \ref{id}.

(ii) Fix $j \geq 0$. By Theorem \ref{bass}(ii)(d) we get that $\mu_j(P, M_c)>0$ if and only if $\mu_j(P, M_{-m})>0$ for any $c \leq -m$. The result follows.

(iii) Follows by Theorem \ref{bass}(ii)(c) with similar arguments as in (ii).

(iv)(a) and (iv)(b) clearly follows from Theorem \ref{bass}.
\end{proof}

\section{Associated Primes}
In this section we maintain our general assumptions and give a sufficient condition under which the collection of all associated primes of any graded component of \LC is finite. We also establish relations between the set of associated primes of graded components under certain condition.

\begin{theorem}[with standard assumption \ref{sa}]\label{finass}
Further assume that $A$ is a regular local domain or a smooth affine algebra over a field of characteristic zero. Then $\bigcup_{n \in \Z}\Ass_B M_n$ is a finite set. 

Moreover, if $B$ is Gorenstein then
\begin{enumerate}[\rm(1)]
\item $\Ass_B M_n= \Ass_B M_{-m}$ for all $n \leq -m$.
\item $\Ass_B M_n= \Ass_B M_0$ for all $n \geq 0$.
\end{enumerate}
\end{theorem}

To prove this theorem we will use the following fact.

\noindent
{\it Observation}: Let $U$ be a commutative Noetherian ring and $G$ be a finite subgroup of the group of automorphisms of $U$ with $|G|$ is invertible in $U$. Let $V$ be the ring of invariants of $G$. Let $T$ be a Lyubeznik functor on $\Mod(V)$. Then we have an $V$-linear Reynolds operator $\rho^U: U \rt V$ which splits the inclusion map $i: V \hookrightarrow U$. Thus $U= V \oplus L$ as $V$-modules. So $T(U)=T(V) \oplus T(L)$ as $V$-modules. It follows that if $\Ass_V T(U)$ is a finite set then so is $\Ass_V T(V)$.

\begin{proof}[Proof of Theorem \ref{finass}]
If $A$ is a smooth affine algebra then so is $S=A[X_1, \ldots, X_m]$. By \cite[Remark 3.7]{Lyu-1} we get $\Ass_S T'(S)$ is a finite set. Since $T'(S)=T(S)$ as $R$-modules so by the above observation it follows that $\Ass_R T(R)$ is a finite set. Again $i: B \hookrightarrow R$ is a ring homomorphism. So by \cite[Proposition 12.1]{TP2} we get $\Ass_B T(R)$ is also a finite set. Moreover, $T(R)= \bigoplus_{n \in \Z} M_n$ as $B$-module. Therefore $\bigcup_{n \in \Z}\Ass_B M_n = \Ass_B T(R)$ is a finite set.
	
Now assume that $A$ is local with maximal ideal $\n$. Since $B= A^G \subset A$ is an integral extension so we have $\m= \n \cap B$ is a maximal ideal of $B$. Now by {\it lying over theorem} for any maximal ideal $\m'$ of $B$ there exists a maximal ideal of $A$ lying over it. Since $A$ has a unique maximal ideal $\n$ so we get $\m'=\n \cap B= \m$. Thus $B$ is also local with unique maximal ideal $\m$. Note $\mathcal{M}=(\m, X_1, \ldots, X_m)$ is the homogeneous maximal ideal of $R$. As $T(R)$ is a graded $R$-module so by \cite[1.5.6]{BH} all its associated primes are homogeneous and so contained in $\mathcal{M}$. Therefore we have an isomorphism $\Ass_R T(R) \to \Ass_{R_{\mathcal{M}}} (T(R)_{\mathcal{M}})$. Now $T(R)_{\mathcal{M}}= H^{i_1}_{I_1R_{\mathcal{M}}}(H^{i_2}_{I_2R_{\mathcal{M}}}(\cdots H^{i_r}_{I_rR_{\mathcal{M}}}(R_{\mathcal{M}}) \cdots)= G(R_{\mathcal{M}})$ for a Lyubeznik functor $G$ on $\Mod(R_{\mathcal{M}})$. Again by \cite[Lemma 4.1]{TP1} we have $R_{\mathcal{M}}=S_{\mathcal{M}}^G= (S_{\mathcal{M}})^G$.  Notice $G(S_{\mathcal{M}})= H^{i_1}_{I_1S_{\mathcal{M}}}(H^{i_2}_{I_2S_{\mathcal{M}}}(\cdots H^{i_r}_{I_rS_{\mathcal{M}}}(S_{\mathcal{M}}) \cdots)= G'(S_{\mathcal{M}})$, where $G'$ is a Lyubeznik functor on $\Mod(S_{\mathcal{M}})$. By \cite[Theorem 3.4]{Lyu-1} we get $\Ass_{S_{\mathcal{M}}} G'(S_{\mathcal{M}})$ is a finite set. Hence $\Ass_{R_{\mathcal{M}}} G(S_{\mathcal{M}})$ is a finite set by \cite[Proposition 12.1]{TP2}. Then by the observation $\Ass_{R_{\mathcal{M}}} G(R_{\mathcal{M}})$ and hence $\Ass_R T(R)$ is a finite set. Therefore again by \cite[Proposition 12.1]{TP2} it follows that $\Ass_B T(R) = \bigcup_{n \in \Z}\Ass_B M_n$ is a finite set.

Now let $B$ be Gorenstein and \[\bigcup_{n \in \Z}\Ass_B M_n = \{P_1,\ldots, P_l\}.\]

(1) Let $P = P_i$ for some $i$. If $r \leq -m$ then by Theorem \ref{bass} we get that $\mu_0(P,M_r) > 0$ if and only if $\mu_0(P,M_{-m}) > 0$. It follows that $P \in \Ass_B M_r$ if and only if $P \in \Ass_B M_{-m}$. Hence the result follows.

(2) Let $P = P_i$ for some $i$. Let $s \geq 0$. Then by Theorem \ref{bass} we get that $\mu_0(P,M_s) > 0$ if and only if $\mu_0(P,M_0) > 0$. It follows that $P \in \Ass_B M_s$ if and only if $P \in \Ass_B M_{0}$. Hence the result follows.
\end{proof}

\section{Infinite generation}
Our aim in this section is to give a sufficient condition under which $M_n$ is not finitely generated as a $B$-module. Since $A$ is a domain so $B$ is also a domain (as $B \subseteq A$ is a sub-ring).

\begin{theorem}[with standard assumption \ref{sa}]
Let $J$ be a homogeneous ideal in $R$ such that $J \cap B \neq 0$. If $B$ is Gorenstein and $H_J^i(R)_c \neq 0$ then $H_J^i(R)_c$ is NOT finitely generated as a $B$-module.
\end{theorem}

\begin{proof}
Set $M= H_J^i(R)$. Notice $R_P= B_P[X_1, \ldots, X_n]$ for any prime ideal $P \in \Spec(B)$. Again $(M_c)_P= H_J^i(R)_c \otimes _B B_P= H_j^i(R \otimes_B B_P)_c= H_J^i(R_P)_c$. Clearly if $H_J^i(R_P)_c$ is not finitely generated then $M_c$ is also not finitely generated. Thus it is enough to prove this result considering $B$ is a local ring. 

We prove by contradiction. If possible let $0 \neq H_J^i(R)_c$ be a finitely generated $B$-module. Then we have $\depth M_c \leq \dim M_c$. Since $B$ is Gorenstein so by Theorem \ref{idsd} we have $\injdim M_c \leq \dim M_c < \infty$. Thus by \cite[Theorem 3.1.17]{BH} we get $\dim M_c \leq \injdim M_c= \depth B$. As $B$ is Cohen Macaulay so $\depth B= \dim B$. Together we have $\dim M_c \leq \dim B= \injdim M_c \leq \dim M_c$, i.e., $\dim M_c= \dim B$. It follows that $M_c$ is maximal Cohen Macaulay $B$-module. Moreover, $B$ is a domain. Therefore $M_c$ is torsion-free. Let $0 \neq a \in J \cap B$. Then $0 \neq a^i \in J \cap B$ for all $i \geq 1$. Clearly $a^i$ is $B$-regular. Since $M_c$ is $J$-torsion and $a \in J$ so we get $(0:_M a^j) \neq 0$ for some $j$, a contradiction. 
\end{proof}

\let\cleardoublepage
\end{document}